\documentclass[11pt]{article}
\usepackage{amsmath, amssymb, amsthm, verbatim,enumerate,bbm, mathtools}
\usepackage{indentfirst}

\date{\today}
\allowdisplaybreaks

\theoremstyle{plain}
\newtheorem{theorem}{Theorem}[section]
\newtheorem{lemma}[theorem]{Lemma}
\newtheorem{claim}[theorem]{Claim}

\newtheorem{observation}[theorem]{Observation}
\newtheorem{corollary}[theorem]{Corollary}

\newtheorem{problem}[theorem]{Problem}
\newtheorem{remark}[theorem]{Remark}

\title{The probability of selecting $k$ edge-disjoint Hamilton cycles in the complete graph}

\author{ Asaf Ferber \thanks{Department of Mathematics, University of California, Irvine. Email: {\tt asaff@uci.edu}. Research is partially supported by an NSF grant DMS-1954395.} \hspace{10mm} Kaarel Haenni \thanks{Massachusetts Institute of Technology. Department of Mathematics. Email: {\tt kaarelh@mit.edu}} \hspace{10mm} Vishesh Jain \thanks{Massachusetts Institute of Technology. Department of Mathematics. Email: {\tt visheshj@mit.edu}}}

\date{}
\begin{document}

\maketitle

\abstract{Let $H_1,\dots,H_k$ be Hamilton cycles in $K_n$, chosen independently and uniformly at random. We show, for $k = o(n^{1/100})$, that the probability of $H_1,\dots,H_k$ being edge-disjoint is $(1+o(1))e^{-2\binom{k}{2}}$}. This extends a corresponding estimate obtained by Robbins in the case $k=2$.  

\section{Introduction}

A classical problem in elementary combinatorics is to show that the number of derangements of an $n$ element set (recall that a derangement is a permutation with no fixed points) is $(1+o(1))\frac{n!}{e}$. This problem can equivalently be formulated in graph theoretic language as follows: let $K_{n,n}$ be the complete bipartite graph with each part of size $n$.  %with both parts $A$ and $B$ be disjoint copies of $[n]:=\{1,\ldots,n\}.$  
Let $M_1$ and $M_2$ be two perfect matchings of $K_{n,n}$, chosen independently and uniformly at random.  %randomly chosen perfect matchings of $K_{n,n}$. 
Then, the probability that $M_1\cap M_2=\emptyset$ is $(1+o(1))\frac 1e.$

This formulation immediately suggests the following question: suppose that $M_1,\ldots,M_k$ are $k$ perfect matchings of $K_{n,n}$, each of which is chosen independently and uniformly at random. What is the probability that all of the $M_i$'s are edge-disjoint? Using (nowadays) standard estimates on the permanent of the (bipartite) adjacency matrix of a $d$-regular bipartite graph, one can readily show that (for $k$ which is not too large compared to $n$), the answer to this question is $(1+o(1))e^{-\binom{k}{2}}$ -- we leave this as an exercise for the interested reader.  %obtain easily obtain that the answer is approximately $(1+o(1))e^{-\binom{k}{2}}$ TO CHECK (for the convenient of the reader we will show it in a later section DO WE WANT TO??).  
Of course, one may ask the same question with perfect matchings replaced by any other graph, and $K_{n,n}$ replaced by some other `ground graph' %which is not necessarily $K_{n,n}.$ That is, the main problem we are interested at is the following: 

\begin{problem}
Let $H$ be a graph on at most $n$ vertices, and let $G$ be a graph that contains at least one copy of $H$. Let $X_1,\ldots,X_k$, $k\geq 2$, be i.i.d random variables, each of which outputs a copy of $H$ in $G$, distributed uniformly at random. What is the probability $p(G,H,k)$ that all the copies $X_i$ are edge-disjoint? 
\end{problem}

For $G = K_n$ (the complete graph on $n$ vertices), $H = C_n$ (a simple cycle on $n$ vertices, also known as a \emph{Hamilton cycle}), and $k=2$, it was shown in \cite{Rob} using a clever inclusion-exclusion argument that $p(G,H,2) = (1+o(1))e^{-2}$. As in the case of perfect matchings, it is natural to ask for $p(G,H,k)$ for $k > 2$. Unfortunately, it seems rather hard to extend the argument of $\cite{Rob}$ to larger values of $k$ (in fact, even a possible extension to $k=3$ seems quite involved). In this note, using a completely different argument, we resolve this problem for all values of $k$ up to some small polynomial in $n$. Specifically, we prove the following theorem.    

%As this problem seems too complicated to solve in full, in this paper we focus on the case where $G=K_n$ (that is, the complete graph on $n$ vertices), and $H=C_n$ (that is, a simple cycle on $n$ vertices, also known as a \emph{hamiltonian cycle}). Note that for such $G$ and $H$, and for $k=2$, it was shown in \cite{Rob}, using a clever inclusion-exclusion estimation, that $p(G,H,2)=(1+o(1))e^{-2}$. Unfortunately, it looks quite hard to extend their approach to larger values of $k$, and even $k=3$ seems quite complicated.  

%In this paper we give a different solution, based on a deterministic algorithm to find all hamiltonian cycles and some concentration of measure techniques, that immediately extends to all values of $k$ up to some small polynomial in $n$.

%Specifically, we prove the following theorem.

\begin{theorem}
  \label{lemma:edge disjoint ham cycles}
Let $k = o(n^{1/100})$, and let $H_1,\ldots, H_k$ be  Hamilton cycles in $K_n$, each of which is chosen independently and uniformly at random. Then, the probability that all the $H_i$'s are edge disjoint is
$$(1+o(1))e^{-2\binom{k}{2}}.$$
\end{theorem}
\begin{remark}
In order to keep the exposition simple, and since our approach anyway does not seem to work for values of $k$ larger than $\sqrt{n}$, we did not make any effort to optimize the upper bound on $k$ in the above theorem.  %since our approach does not seem to work for values of $k$ which are larger than (say) $\sqrt{n}$ 
\end{remark}
%TO CHECK WHAT IS THE NATURAL BARRIER

\subsection{Outline of the proof}
By Bayes' rule, it suffices (see Section~\ref{sec:proof-main} for details) to show that the number of Hamilton cycles in any graph obtained by removing $i$ edge-disjoint Hamilton cycles from $K_n$ is $(1+o(1/k))\cdot e^{-2i}\cdot (n-1)!/2$ -- this is the content of our main technical lemma (Lemma~\ref{lemma:main}). Our proof of this lemma consists of providing an algorithm to generate all Hamilton cycles in a given graph $G$ (see Section~\ref{procedure}), and then using standard estimates on the number of perfect matchings in bipartite graphs (Corollary~\ref{Bregman:cor}, Lemma~\ref{lemma:lower bound on perfect matchings}), as well as standard concentration inequalities (Theorem~\ref{Azuma for permutations}), in order to analyze the number of distinct Hamilton cycles our algorithm can output.

Roughly speaking, our algorithm generates Hamilton cycles as follows: for a sufficiently large integer $\ell$, divide the vertices of $G$ into $\ell$ parts of size $n/\ell$ (for the sake of this discussion, we assume that $n$ is divisible by $\ell$); choose a perfect matching between parts $i$ and $i+1$ for $1\leq i \leq \ell - 1$ to obtain a collection of $n/\ell$ (oriented) paths of length $\ell - 1$, and finally, extend (if possible), this collection of oriented paths to an oriented Hamilton cycle of $G$. In Claim~\ref{claim: number of edges}, we show using a standard concentration of measure argument that most ways to partition the vertices satisfy a certain `goodness' property -- the contribution to our enumeration coming from partitions not satisfying this property is so small that it may be ignored. On the other hand, for partitions satisfying this goodness property, we are able to effectively leverage standard estimates on the number of perfect matchings in bipartite graphs to provide an asymptotically correct estimate of the number of choices available to our algorithm in subsequent steps.

\section{Tools and auxiliary results}

In this section, we collect some tools and auxiliary results to be used in the proof of our main result. 

\subsection{McDiarmid's inequality} 

We will make use of the following concentration inequality due to McDiarmid (see \cite{Mac}, Section 3.2).

\begin{theorem} \label{Azuma for permutations}
Let $S_n$ denote the symmetric group on $n$ elements and let
$f:S_n\rightarrow \mathbb{R}$ be such that $|f(\pi)-f(\pi')|\leq u$
whenever $\pi'$ can be obtained from $\pi$ by a single transposition. If $\pi$ is chosen uniformly at random from $S_n$, then

$$\Pr\left[|f(\pi)-\mathbb{E}(f)|\geq t\right]\leq
2\exp\left(-\frac{2t^2}{nu^2}\right).$$
\end{theorem}

\subsection{A procedure to find all Hamilton cycles in a graph $G$}\label{procedure}

In this section, we describe a procedure to find all Hamilton cycles in a given graph $G$. Later, in Section~\ref{subsec:counting}, we will estimate (from below and above) the number of distinct Hamilton cycles that this procedure can output. 

\begin{enumerate}
\item Fix any positive integer $\ell$ (possibly depending on $n$). Let $r=n \mod \ell$, so that $0\leq r\leq \ell-1$. Let $[n]=V_1\cup \ldots\cup V_\ell$ be any partition of $[n]$, where the first $r$ parts have size $t_c:=\left\lceil{\frac{n}{\ell}}\right\rceil$, and the last $\ell-r$ parts have size $t_f:= \left\lfloor{\frac{n}{\ell}}\right\rfloor$.
\item If $r\neq 0$, designate a `root' $v^{\ast}\in V_{r}$. If $r = 0$, designate a `root' $v^{\ast}$ in $V_1$.  
\item For each $j\in [1,r-1]\cup [r+1,\ell-1]$, let $B_{j}:=G[V_j,V_{j+1}]$. If $r\neq 0$, let $B_r:=G[V_r\setminus \{v^{\ast}\},V_{r+1}]$ and $B_\ell:=G[V_\ell\cup \{v^{\ast}\},V_1]$; if $r = 0$, let $B_\ell:=G[V_{\ell},V_1]$.
\item For each $1\leq j\leq \ell-1$, choose a perfect matching $M_j$ of $B_j$, and observe that $\cup_j M_j$  is a collection of $t_c$ vertex disjoint paths, of which $t_f$ have length $\ell$ and $t_c - t_f$ have length $r$ (by the length of a path, we mean the number of vertices in it). Let $\mathcal P:=\{P_1,\ldots,P_{t_c}\}$ denote the obtained collection of paths, and orient each path such that the vertex in $V_1$ is the first vertex. %Consider the natural orientation for each path (that is, each path is oriented so that vertex in $V_1$ is $1^{st}$ and the vertex in $V_{\ell}$ (or in $V_r$ for the paths of length $r$) is last).
\item Finally, using only the edges in $B_{\ell}$ (directed from $V_\ell$ to $V_1$), find (if possible) a rooted, oriented Hamilton cycle in $G$, which is rooted at $v^{\ast}$ and contains all the paths in $\mathcal P$ as oriented segments.
\end{enumerate}

Let $\mathcal{H}_{r,o}(\ell)$ denote the collection of rooted, oriented Hamilton cycles in $G$ obtained by running the above procedure (with some fixed positive integer $\ell$) for all possible choices of partitions in Step 1, all possible choices of the root in Step 2, all possible choices of the perfect matchings in Step 4, and all possible choices of the compatible rooted, oriented Hamilton cycle in Step 5.   

%Observe that indeed, every Hamilton cycle of $G$ can be an output of this procedure exactly $2n$ times. 

\begin{lemma}
 For every positive integer $\ell$, the collection $\mathcal{H}_{r,o}(\ell)$ contains every rooted, oriented Hamilton cycle of $G$ exactly once. 
\end{lemma}

\begin{proof}
 Fix a rooted, oriented Hamilton cycle $H$ in $G$. There is exactly one partition of the vertices in Step 1 compatible with $H$ -- indeed, the root $v^{\ast}$ must belong to $V_{\max\{r,1\}}$, and following the Hamilton cycle from the root along its orientation determines the partition of the vertices. Once this is done, note that the choice of perfect matchings (equivalently, the collection of oriented paths $\mathcal{P}$) in Step 4 is automatically determined by the edges present in the Hamilton cycle. Finally, given this collection of paths, there is exactly one choice of edges in Step 5 which is compatible with $H$.   
\end{proof}

Let $\mathcal{H}(\ell)$ be the collection of Hamilton cycles in $G$, obtained from $\mathcal{H}_{r,o}(\ell)$ by forgetting the root and the orientation. Since for any Hamilton cycle in $G$, there are exactly $n$ ways to choose a root for it, and exactly $2$ ways to orient it, we have:
\begin{observation}
 For every positive integer $\ell$, the collection $\mathcal{H}(\ell)$ contains every Hamilton cycle of $G$ exactly $2n$ times. 
\end{observation}

%\begin{proof}
%Indeed, given any Hamilton cycle of $G$, there are $2$ ways to choose an orientation for it, and $n$ ways to choose the designated vertex. A moment's thought now reveals that this uniquely determines the partition and the matchings in between.  
%\end{proof}

\subsection{The number of perfect matchings in bipartite graphs}

In order to estimate the number of Hamilton cycles our procedure can output, we will need to estimate the number of perfect matchings in `typical' bipartite graphs obtained by our procedure.\\ %Since our procedure to find Hamilton cycles is based on finding perfect matchings in certain bipartite graphs, 

For bounding the number of perfect matchings in a bipartite graph from above, we use the following theorem due to Br\'egman  (see e.g. \cite{AlonSpencer}, page 24) that relates the number of perfect matchings to the vertex-degrees in the graph. 

\begin{theorem}\label{bregman}(Br\'egman's Theorem)
Let $G=(A\cup B,E)$ be a bipartite graph with both parts of the same size. Then, the number of perfect matchings in $G$ is at most
$$\prod_{a\in A} (d_G(a)!)^{1/d_G(a)}.$$
\end{theorem}

The following is an immediate corollary of Theorem \ref{bregman} and Stirling's approximation.

\begin{corollary}
  \label{Bregman:cor}
  Let $H$ be a spanning subgraph of $K_{m,m}$ with maximum degree at most $D \leq m/2$. Then, the number of perfect matchings in $K_{m,m}\setminus H$ is at most
%  $$ \left(1+O\left(\frac{D\log m}{m}\right)\right)m!e^{-|E(H)|/m}.$$
$$e^{O(D\log{m}/m)}\cdot m!\cdot e^{-|E(H)|/m}.$$
\end{corollary}

\begin{proof} By applying Theorem \ref{bregman} to $G:= K_{m,m}\setminus H$ and using the fact that $s!=(1+O(1/s))\sqrt{2\pi s}\left(\frac{s}{e}\right)^s$, one obtains that the number of perfect matchings in $K_{m,m}\setminus H$ is at most $$\left(\prod_{a\in A}\left(1+O\left(\frac{1}{m-D}\right)\right)^{1/(m-D)}\right)\cdot \left(\prod_{a\in A}(\sqrt{2\pi m})^{1/(m-D)}\right)\cdot \left(\prod_{a \in A}\left(\frac{m-d_H(a)}{e}\right)\right).$$
Using the assumption $D \leq m/2$, the first term can be estimated by:
\begin{align*}
\prod_{a\in A}\left(1+O\left(\frac{1}{m-D}\right)\right)^{1/(m-D)}
&\leq e^{O(m/(m-D)^{2})}\\
&\leq e^{O(1/m)}.
\end{align*}
The second term can be estimated by:
\begin{align*}
    \prod_{a\in A}(\sqrt{2\pi m})^{1/(m-D)} &=
    \sqrt{2\pi m}\cdot (\sqrt{2 \pi m})^{D/(m-d)}\\
    &\leq \sqrt{2\pi m}\cdot e^{O(D\log{m}/(m-D))}\\
    &\leq \sqrt{2\pi m}\cdot e^{O(D\log{m}/m)}.
\end{align*}
The third term can be estimated by:
\begin{align*}
\prod_{a \in A}\left(\frac{m}{e}\cdot\left(1-\frac{d_H(a)}{m}\right) \right)
&\leq \left(\frac{m}{e}\right)^{m}\cdot \prod_{a \in A}\exp\left(-\frac{d_{H}(a)}{m}\right)\\
&\leq \left(\frac{m}{e}\right)^{m}\cdot \exp\left(-\frac{|E(H)|}{m}\right).
\end{align*}
Combining everything, and using Stirling's approximation once again, we get the upper bound:
\begin{align*}
    e^{O(D\log{m}/m)}\cdot \left(\sqrt{2 \pi m}\left(\frac{m}{e}\right)^{m}\right)\cdot e^{-|E(H)|/m}
    &\leq e^{O(D\log{m}/m)}\cdot m!\cdot e^{-|E(H)|/m},
\end{align*}
as desired. \qedhere 
% \begin{align}(2\pi m)^{m/2(m-D)}\left(\frac{m}{e}\right)^me^{-|E(H)|/m}= \left(1+O\left(\frac{D\log m}{m}\right)\right)m!e^{-|E(H)|/m}, \nonumber
% \end{align}
% as desired.
\end{proof}

%To obtain a lower bound on the number of Hamiltonian cycles in a graph $G$ with ``not too many" missing edges (that is, after some $t<k$ hamiltonian cycles have already been removed from $K_n$), we will run the above procedure with $\ell$ which is relatively large. Therefore, the graphs $B_j$ which are defined in Step 3 of the procedure are relatively small and therefore, at least in expectation, are very close to be complete bipartite graphs. Therefore, the following lower bound on the number of perfect matchings in nearly complete bipartite graphs is useful to us:  

The next lemma provides a nearly matching lower bound on the number of perfect matchings in `almost complete' balanced bipartite graphs. 
\begin{lemma}
  \label{lemma:lower bound on perfect matchings}
    Let $H$ be a spanning subgraph of $K_{m,m}$  with $|E(H)|<\frac{m}{4}$. Then, the number of perfect matchings in $K_{m,m}\setminus H$ is at least 
  $$\left(1-O\left(\frac{|E(H)|^2}{m^2}\right)\right)\cdot m!\cdot e^{-|E(H)|/m}.$$
\end{lemma}

\begin{proof}
  Represent $G:= K_{m,m}\setminus H$ as $G = (A\cup B, E)$,  %is a bipartite graph with with parts $A\cup B$. 
  and label the vertices of $A=\{v_1,\ldots,v_{m}\}$ in such a way that all the $q\leq m/4$ vertices in $A$, which are not isolated in $H$, are labeled as $v_1, v_2, \ldots, v_q$. For each $1\leq i\leq q$, let $d_i:=d_H(v_i)$. 
  
  We will construct perfect matchings of $G$ by manually pairing each vertex in $A$ with a vertex in $B$. For this, note that there are at least $m-d_1$ ways to choose a vertex in $B$ to pair with $v_1$.  Having chosen such a vertex, there are at least $m-d_2-1$ ways to choose a vertex in $B$, different from the one chosen in the previous step, to pair with $v_2$. In general, for $1\leq i \leq q$, there are at least $m-d_i-(i-1)$ ways to choose a vertex in $B$, different from the ones chosen in the first $i-1$ steps, to pair with $v_i$. Having matched the first $q$ vertices, note that all of the remaining vertices in $A$ have edges (in $G$) to all of the vertices in $B$ and hence, the number of ways in which we can find a vertex in $B$ to pair with $v_i$ for $i>q$ is exactly $m-(i-1)$. Since each sequence of choices gives a different perfect matching, it follows that the number of perfect matchings in $G$ obtained in this manner is at least 
  \begin{align*} \left(\prod_{i=1}^q \left(m-d_i-(i-1)\right)\right) \prod_{i=q+1}^m \left(m-(i-1)\right)
&=m!\cdot \prod_{i=1}^q \frac{m-d_i-i+1}{m-i+1}\\
&=m! \cdot \prod_{i=1}^q \left(1-\frac{d_i}{m-i+1}\right)\\
&\geq m! \left(1- \sum_{i=1}^q \frac{d_i}{m-i+1}\right)\\
&\geq m!  \left(1- \sum_{i=1}^q \frac{d_i}{m-q}\right)\\
&= m! \left(1-\frac{|E(H)|}{m-q}\right),
\end{align*}
where the third line uses the elementary inequality $\prod_{i=1}^{n} (1-x_i)\geq 1-\sum_{i=1}^{n} x_i$, valid for $x_1,\dots,x_n \geq 0$. Next, using the numerical inequality $(1-x)^{-1} \leq 1+2x$, valid for $x \in [0,1/2]$, we have
\begin{align*}
%&\geq m! \left(1- \sum_{i=1}^q \frac{d_i}{m-i+1}\right)\geq m! \left(1-\sum_{i=1}^q \frac{d_i}{m-q}\right)\\
m! \left(1-\frac{|E(H)|}{m-q}\right)
%&=m!\cdot \left(1- \frac{|E(H)|}{m}\frac{1}{1-\frac{q}{m}}\right)\\
&\geq m! \left(1- \frac{|E(H)|}{m}\cdot\left(1+\frac{2q}{m}\right)\right)\\
%&= m! \left(1-\frac{|E(H)|}{m}\right)\frac{m^2-|E(H)|m-2q|E(H)|}{m^2-|E(H)|m}  \\
&= m!\left(1-\frac{|E(H)|}{m}\right)\left(1-\frac{2q|E(H)|}{m(m-|E(H)|)}\right)\\
&\geq m!\left(1-\frac{|E(H)|}{m}\right)\left(1-3\frac{|E(H)|^2}{m^2}\right),
\end{align*}
where the last equality uses that $q \leq |E(H)| < m/4$. 
%We continue by using $\frac{1}{1-x}\leq 1+2x$ for $0\leq x\leq \frac{1}{2}$:
Finally, using the numerical inequality $1-x \geq e^{-x}(1-x^2)$, valid for %e^{-x}-\frac{x^2}{2}= e^{-x}\left(1-\frac{x^2}{2e^{-x}}\right)\geq e^{-x}\left(1-\frac{x^2}{2(1-x)}\right)\geq e^{-x}(1-x^2)$ for 
$x \in [0,1]$, we can bound the right hand side from below by
\begin{align*}
m!\cdot e^{-|E(H)|/m} \left(1- \frac{|E(H)|^2}{m^2}\right)\left(1-3\frac{|E(H)|^2}{m^2}\right) &\geq m!\cdot e^{-|E(H)|/m} \left(1-10\frac{|E(H)|^2}{m^2}\right)\\
&=\left(1-O\left(\frac{|E(H)|^2}{m^2}\right)\right)m!\cdot e^{-|E(H)|/m}.
\qedhere 
\end{align*}
%This completes the proof. 
\end{proof}

\subsection{The number of Hamilton cycles obtained by our procedure}
\label{subsec:counting}

In this section, we present the key step in the proof of our main theorem -- a near-optimal estimate on the number of Hamilton cycles in a graph $G$ obtained by deleting $i<k$ edge-disjoint Hamilton cycles from $K_n$. Specifically, we prove the following lemma: 

\begin{lemma}\label{lemma:main}

Let $k = o(n^{1/100})$, and let  $H_1,\ldots,H_k$ be i.i.d. random variables, each of which outputs a Hamilton cycle of $K_n$, chosen uniformly at random. For each $1\leq i\leq k$, let $\mathcal E_i$ be the event ``$E(H_i)\cap \left(\cup_{j<i}E(H_j)\right)=\emptyset$" i.e. no edge of $H_i$ is present in $H_j$, for any $j < i$. Then, for every $0\leq i\leq k-1$
  \begin{align*} \Pr\left[\mathcal E_{i+1} \mid \mathcal E_i\cdots \mathcal E_1\right]= \exp\left(\pm O\left(\frac{1}{k^{2}\log^{4}n}\right)\right)\cdot\exp(-2i).
\end{align*}
\end{lemma}

To prove this lemma, we will analyze the procedure for generating all Hamilton cycles of $G$ given in Section \ref{procedure}. We will need the following two preliminary claims.\\ 

%Let us first state a few results that help with the counting. 
The first claim concerns the number of partitions in Step 1 of Section~\ref{procedure}. %is a simple counting of the number of partitions.

\begin{claim}
 The number of partitions $\mathcal{V}$ of $[n]$ into $r$ sets $V_1, \ldots, V_r$ of size $t_c$ and $\ell-r$ sets $V_{r+1},\ldots V_\ell$ of size $t_f$, together with a designated vertex $v^{\ast}\in V_{\max\{r,1\}}$ is $\frac{n!\cdot t_c}{(t_c!)^r(t_f!)^{\ell-r}}$.
\end{claim}

\begin{proof}
Indeed, there are $\frac{n!}{(t_c!)^{r}(t_f!)^{\ell-r}}$ ways of choosing an (ordered) partition with the given sizes, and $t_c$ ways of choosing a designated vertex from $V_{\max\{r,1\}}$. %so there are $\frac{n!t}{(t!)^r\left((t-1)!\right)^{\ell-r}}$ ways.
\end{proof}

Let $G$ be a graph obtained from $K_n$ by removing $i$ edge-disjoint Hamilton cycles, and fix $\ell \leq \sqrt{n}$. For a partition of $[n]$ into $\ell$ parts as above, let $\{B_j\}_{1\leq j \leq \ell}$ denote the collection of bipartite graphs constructed in Step 3 of Section~\ref{procedure}. We claim that, for most partitions, the number of edges missing from each $B_j$ is close to its expectation (for a uniformly random partition).

\begin{claim} \label{claim: number of edges} %Suppose $\ell\leq \sqrt{n}$. 
Let $f_j$ be the number of missing edges in $B_j$. Then, for all sufficiently large $n$, the number of partitions $\mathcal{V}$ of $[n]$ for which $\lvert f_j-\frac{2\cdot i\cdot n}{\ell^2}\rvert \leq 2n^{2/3}$ for every $1\leq j\leq \ell$ is at least
 $$(1-e^{-n^{3/100}})\frac{n!}{(t_c!)^{r}(t_f!)^{\ell-r}}.$$
\end{claim}

\begin{proof}
 Let $\sigma\in S_n$ be a uniformly random permutation of $[n]$. Let $V_1$ be the image of $[1,|V_1|]$ under $\sigma$, $V_2$ be the image of the next $|V_2|$ elements in $[n]$ under $\sigma$, and so on. %and define \[V_j:=\{\sigma((j-1)t+1),\ldots, \sigma(jt)\} \text{ for every }j\in [1,r]\] \[V_j:=\{\sigma(rt+(j-r-1)(t-1)+1),\ldots, \sigma(rt+(j-r)(t-1))\} \text{ for every }j\in [r+1,\ell] .\]
For $1\leq j\leq \ell$, let $f_j$ be the number of missing edges in $B_j$, and observe that
$$\mu_j:= \mathbb{E}f_{j} = \frac{2\cdot i\cdot n}{\ell^{2}} \pm O\left(\frac{i}{\ell}\right) = \frac{2\cdot i\cdot n}{\ell^{2}} \pm o(n^{2/3}),$$
where the final inequality uses the assumption $i\leq k = o(n^{1/100})$.
%$$\mu_j:=\mathbb{E}f_j= \begin{cases}  \frac{2int^2}{n(n-1)} &\mbox{if } j\in [1,r-1] \cup \{\ell\} \\
%\frac{2in(t-1)^2}{n(n-1)} & \mbox{if }j\in [r, \ell-1]. \end{cases} $$
%  \begin{align*}
%      &\mu_j:=\mathbb{E}f_j=\frac{2int^2}{n(n-1)} \text{ for } j\in [1,r-1] \text{ and } j=\ell \\
%      & \mu_j:=\mathbb{E}f_j=\frac{2in(t-1)^2}{n(n-1)}\text{ for } j\in [r, \ell-1].
%  \end{align*}
Next, since $\cup_{j\leq i}E(H_i)$ is a $2i$-regular graph, it follows that a single transposition of $\sigma$ can change $f_j$ by at most $2i$.    
 %and that by transposing two elements of $\sigma$ one can change $f_j$ by at most $2i$ (since $\cup_{r\leq i}E(H_r)$ is a $2i$-regular graph).
 Therefore, by Theorem \ref{Azuma for permutations}, %using the fact that $i\leq k=o(n^{1/100})$ and that $\ell\leq n$, one obtains that
 $$\Pr[|f_j-\mu_j|\leq n^{2/3}]\leq 2e^{-2n^{4/3}/(4\cdot n\cdot i^{2})}=o(e^{-n^{3/100}}/\ell),$$
 where the final inequality uses the assumption $i\leq k = o(n^{1/100})$. Applying the union bound for $1\leq j \leq \ell$ shows that the number of permutations giving rise to `good partitions' (i.e. those satisfying the assumptions of the lemma) is at least $\left(1-e^{-n^{3/100}}\right)n!$. Finally, since each partition corresponds to $(t_c !)^{r} (t_f !)^{\ell - r}$ distinct permutations, we get the desired conclusion. \qedhere
% By applying the union bound over all $j\leq \ell$, we obtain that the same concentration inequality holds for all $j$ simultaneously with probability at least $1-e^{-n^{3/100}}$. Note that we are done with proving the lemma if we show that $|\mu_j-\frac{2in}{\ell^2}|\leq n^{2/3}$ for $n$ large enough. To see this, we wish to show that $f_j-o(n^{2/3})\leq \mu_j\leq f_j+o(n^{2/3})$. For the upper bound, recall that $\ell \leq \sqrt{n}$ and consider the following:
% \[\mu_j\leq \frac{2int^2}{n(n-1)} = \frac{2in\left\lceil{\frac{n}{\ell}}\right\rceil^2}{n(n-1)}\leq \frac{2in \left(\frac{n}{\ell}+1\right)^2}{n(n-1)}=\frac{2i\left(\frac{n}{\ell}\right)^2}{n}\left(1+O(1/n)\right)\left(1+O(\ell/n)\right)^2\]
 
% \[=\frac{2in}{\ell^2}\left(1+O\left(n^{-1/2}\right)\right)=\frac{2in}{\ell^2}+o(n^{2/3}).\]
 
% The lower bound on $\mu_j$ is obtained similarly. This completes the proof.
\end{proof}

With these two claims in hand, we can prove Lemma \ref{lemma:main}.

\begin{proof}[Proof of Lemma \ref{lemma:main}]
  Let $H_1,\ldots, H_i$ be any edge-disjoint Hamilton cycles in $K_n$, and let $G$ be the graph obtained from $K_n$ by removing $\cup _{j \leq i}E(H_j)$. We wish to count the number of Hamilton cycles in $G$, and we will do so by analyzing the procedure in Section~\ref{procedure}. %$H_{i+1}$ which are edge-disjoint to $\cup_{j\leq i}E(H_j)$. 
  %We do this counting by applying the procedure described before to all partitions $V_1\cup \ldots \cup V_\ell$ (where all parts have the right size) and both upper and lower bounding the number of cycles attained in this way. We 
  For the rest of this proof, we fix  $\ell=\lceil (k\log n)^4\rceil = o(n^{1/20})$.
  
  First, note that the number of Hamilton cycles in $G$ obtained by our procedure, starting from a partition in Step 1 which does not satisfy the conclusion of Claim~\ref{claim: number of edges} is negligible for our purposes. Indeed, by Claim~\ref{claim: number of edges}, the number of such partitions is at most $e^{-n^{3/100}}\cdot n!/((t_c !)^{r}\cdot (t_f !)^{\ell - r})$, and once we fix such a partition, the number of choices available in Steps 2-5 is at most $t\cdot (t_c !)^{r}\cdot  (t_f!)^{\ell - r}$. Hence, the number of Hamilton cycles that can be obtained in this manner is at most $t\cdot n!\cdot e^{-n^{3/100}} = o(n!\cdot e^{-2i}/\text{poly}(n))$, for $i\leq k= o(n^{1/100})$. Therefore, it suffices to analyze the contribution of partitions satisfying the conclusion of Claim~\ref{claim: number of edges}. 
  
  For any such partition $V = V_1 \cup \dots \cup V_{\ell}$, there are exactly $t_c$ choices in Step 2.
  
  Moreover, for each such realisation of Steps 1-3, by Corollary~\ref{Bregman:cor} and Claim~\ref{claim: number of edges}, the number of perfect matchings of $B_j$ is at most either (depending on the value of $j$)
  $$t_c!\cdot\exp\left(-\frac{2\cdot i}{\ell}\right)\cdot \exp(O(\ell\cdot n^{-1/3})),$$
  or the same expression with $t_c$ replaced by $t_f$. Similarly, by Lemma~\ref{lemma:lower bound on perfect matchings} and Claim~\ref{claim: number of edges}, the number of perfect matchings of $B_j$ is at least either (depending on the value of $j$)
  $$t_c! \cdot \exp\left(-\frac{2\cdot i}{\ell}\right)\cdot \exp\left(-O\left(\frac{k^{2}}{\ell^{2}}\right)\right),$$
  or the same expression with $t_c$ replaced by $t_f$. 
  %and we start by upper bounding the number of distinct collections of paths obtained by the procedure from a fixed partition $\mathcal{V}$ for which the conclusion of Claim \ref{claim: number of edges} holds. 
  
%  First, note that by applying Corollary \ref{Bregman:cor} (with $D=2k$) to $B_j$ being $K_{t,t}\setminus H$ for all $j\in [1,r-1]\cup \{\ell\}$ and $K_{t-1,t-1}\setminus H$ for all $j\in [r,\ell-1]$ using the conclusion of Claim \ref{claim: number of edges}, one obtains that there are at most

%\begin{align*}
%    &\left(1+O\left(\frac{2k\log t}{t}\right)\right)^{r-1}(t!)^{r-1}e^{-((r-1)2 in/\ell^2-(r-1)2 n^{2/3})/t}\\
%    &\times \left(1+O\left(\frac{2k\log (t-1)}{t-1}\right)\right)^{\ell-r}((t-1)!)^{\ell-r}e^{-(2 in(\ell-r)/\ell^2-(\ell-r)2 n^{2/3})/(t-1)}
%\end{align*}
% distinct collections of paths obtained from the above procedure. Simplifying a bit, this is at most
 
% \begin{align*}
%     &\left(1+O\left(\frac{2k\ell^2\log n}{n}\right)\right)(t!)^{r-1}\left((t-1)!\right)^{\ell-r}e^{-(2 in/\ell^2-2 n^{2/3})\left(\frac{r-1}{t}+\frac{\ell-r}{t-1}\right)}\\
%     &=\left(1+O\left(\frac{2k\ell^2\log n}{n}\right)\right)(t!)^{r-1}\left((t-1)!\right)^{\ell-r}e^{-(2 in/\ell^2-2 n^{2/3})\left(\frac{\ell^2}{n}+O\left(\frac{\ell}{n}\right)\right)}\\
%     &=\left(1+O\left(\frac{2k\ell^2\log n}{n}\right)\right)(t!)^{r-1}\left((t-1)!\right)^{\ell-r}e^{-(2 i-2\ell^2n^{-1/3})+O\left(\frac{i}{\ell}+\ell n^{-1/3} \right)}.
%     \end{align*}
Since there are $r-1$ values of $j$ for which the above bounds hold with $t_c$, and $\ell - r$ values of $j$ for which the above bounds hold with $t_f$, and since $k^{2}/\ell^{2} \gg \ell \cdot n^{-1/3}$, it follows that the number of collection of paths that can be obtained at the end of Step 4 is
$$(t_c!)^{r-1}\cdot(t_f!)^{\ell-r}\cdot \exp\left(-2i\right)\cdot\exp\left(\pm O\left(\frac{k^2}{\ell}\right)\right). $$
%     Plugging in our value of $\ell$,
     
%     \begin{align*}
%     &=\left(1+O\left(\frac{2k\ell^2\log n}{n}\right)\right)(t!)^{r-1}\left((t-1)!\right)^{\ell-r}e^{-2i+O\left(\frac{1}{k^3\left(\log n\right)^4}\right)}\\
%     &\leq \left(1+O\left(\frac{1}{k^3\left(\log n\right)^4}\right)\right)(t!)^{r-1}\left((t-1)!\right)^{\ell-r}e^{-2i}.
% \end{align*}
 
Finally, let us estimate the number of ways to extend any such collection of paths into a Hamilton cycle in Step 5. For this, we arbitrarily label the collection of paths obtained at the end of Step 4 by $1,\dots,t_{c}$. This induces a natural labeling (given by which path the vertex participates in) of each part of the bipartite graph $B_\ell$ by the labels $1,\dots,t_{c}$. The obtained labelled bipartite graph may be viewed as a directed graph (possibly with self loops) on $t_c$ vertices as follows: identify vertices with the same label, and orient edges from the first part of $B_\ell$ to its second part. Observe that the the number of extensions available in Step 5 correspond precisely to the number of oriented Hamilton cycles of this directed graph. 

Since the complete directed graph on $t_c$ vertices has at most $(t_c-1)!$ oriented Hamilton cycles, it follows that there are at most $(t_c-1)!$ such extensions. For a nearly matching lower bound, we begin by noting that the number of oriented Hamilton cycles in the complete directed graph on $t_c$ vertices containing a specific edge is at most $(t_c - 2)!$. Since, by Claim~\ref{claim: number of edges} the directed graph corresponding to $B_\ell$ has at most $O(k\cdot n/\ell^{2})$ missing edges, it follows that the number of oriented Hamilton cycles in this directed graph is at least
$$(t_c - 1)! - O\left(\frac{k\cdot n}{\ell^2}\right)(t_c - 2)! = (t_c - 1)!\left(1 - O\left(\frac{k}{\ell}\right)\right).$$

To summarize, we have shown that the number of choices available in Steps 2-5, for any fixed choice of partition in Step 1 which satisfies the conclusion of Claim~\ref{claim: number of edges} is
\begin{align*}
    t_c \cdot (t_c!)^{r-1}\cdot (t_f!)^{\ell - r} \cdot \exp\left(-2i\right)\cdot (t_c-1)!\exp\left(\pm O\left(\frac{k^{2}}{\ell}\right)\right) = (t_c!)^{r}\cdot(t_f!)^{\ell - r}\cdot \exp(-2i)\cdot\exp\left(\pm O \left(\frac{k^{2}}{\ell}\right)\right).
\end{align*}
Combining this with the number of choices in Step 1, as given by Claim~\ref{claim: number of edges}, we see that the contribution to the number of oriented, rooted Hamilton cycles from such partitions is
$$n! \cdot \exp(-2i) \cdot \exp\left(\pm O\left(\frac{k^2}{\ell}\right)\right) = 2n \cdot \frac{(n-1)!}{2}\cdot \exp(-2i)\cdot \exp\left(\pm O\left(\frac{1}{k^{2}\log^{4}n}\right)\right)$$

To conclude, recall that every (undirected, unrooted) Hamilton cycle is counted exactly $2n$ times by our procedure, and that $K_n$ has $(n-1)!/2$ Hamilton cycles. \qedhere 

\end{proof}

\section{Proof of the main theorem}
\label{sec:proof-main}

Now we are ready to prove Theorem \ref{lemma:edge disjoint ham cycles}.

\begin{proof}[Proof of Theorem~\ref{lemma:edge disjoint ham cycles}]

The proof is a straightforward application of Lemma \ref{lemma:main}.
%As a reminder, we have $H_1,\ldots,H_k$ independent random variables, each of which outputs a Hamiltonian cycle of $K_n$, chosen uniformly at random. For each $1\leq i\leq k$, $\mathcal E_i$ is the event ``$E(H_i)\cap \left(\cup_{j<i}E(H_j)\right)=\emptyset$". Then by Lemma \ref{lemma:main}, for every $0\leq i\leq k-1$
%  \begin{align}\label{main inequality}\Pr\left[\mathcal E_{i+1} \mid \mathcal E_i\cdots \mathcal E_1\right]= \left(1+O\left(\frac{1}{k^2\log^2 n}\right)\right)e^{-2i}.
%\end{align}
Indeed, we wish to find the probability that all $k$ of the chosen Hamiltonian cycles are edge disjoint, which is 
%Using the fact that $1-x\leq e^{-x}$ holds for every $x>0$ and that $\sum_{a\in A}d_H(a)=|E(H)|$, we obtain that the above expression is

\begin{align*}\Pr\left[\mathcal E_1\mathcal E_2\cdots \mathcal E_k\right] %\Pr\left[\mathcal E_1\right]\Pr\left[\mathcal E_2\cdots \mathcal E_k\mid \mathcal E_1\right]\nonumber %\\
&=\Pr\left[\mathcal E_1\right]\Pr\left[\mathcal E_2\mid \mathcal E_1\right]\dots\Pr\left[ \mathcal E_k\mid \mathcal E_{k-1}\dots \mathcal E_2\mathcal E_1\right]\nonumber \\
&= \prod_{i=0}^{k-1}\Pr\left[\mathcal E_{i+1}\mid \mathcal E_i\cdots \mathcal E_1\right]\nonumber\\
&= \prod_{i=0}^{k-1}\exp\left(\pm O\left(\frac{1}{k^{2}\log^{4}n}\right)\right)\cdot \exp(-2i)\\ &=
\exp\left(\pm O\left(\frac{1}{k\log^{4}n}\right)\right)\exp\left(-2\sum_{i=0}^{k-1} i\right)\\
%\left(1+O\left(\frac{1}{k^2\log^2 n}\right)\right)^{k}e^{-2\sum_{i=0}^{k-1}i}\nonumber \\
%&=\left(1+O\left(\frac{1}{k\log n}\right)\right)e^{-2\binom{k}{2}}\nonumber \\
&=(1+o(1))e^{-2\binom{k}{2}},
\end{align*}
where the third line uses Lemma~\ref{lemma:main}. \qedhere

\end{proof}

\end{document}